\documentclass[11pt,reqno]{amsart}

\usepackage{amssymb}

\usepackage{amsfonts}

\usepackage{amsmath}

\usepackage[all]{xy}

\newtheorem{theorem}{Theorem}[section]

\newtheorem{corollary}[theorem]{Corollary}

\newtheorem{lemma}[theorem]{Lemma}

\newtheorem{proposition}[theorem]{Proposition}

\newtheorem{Definition}[theorem]{Definition}

\newtheorem{Example}[theorem]{Example}

\newtheorem{Remark}[theorem]{Remark}

\newenvironment{remark}{\begin{Remark}\begin{em}}{\end{em}\end{Remark}}
\newenvironment{example}{\begin{Example}\begin{em}}{\end{em}\end{Example}}
\newenvironment{definition}{\begin{Definition}\begin{em}}{\end{em}\end{Definition}}

\newcommand{\uu}{\mathbf{u}}

\newcommand{\vv}{\mathbf{v}}

\DeclareMathOperator{\tr}{tr}

\DeclareMathOperator{\gyr}{gyr}

\address{Sejong Kim, Department of Mathematics, Chungbuk National University, Cheongju 28644, Korea}
\email{skim@chungbuk.ac.kr}

\begin{document}

\author{Sejong Kim}

\title[Operator inequalities and gyrolines of the weighted geometric means]{Operator inequalities and gyrolines of the weighted geometric means}

\date{\today}
\maketitle

\begin{abstract}
We consider in this paper two different types of the weighted geometric means of positive definite operators. We show the component-wise bijection of these geometric means and give a geometric property of the spectral geometric mean as a metric midpoint. Moreover, several interesting inequalities related with the geometric means of positive definite operators will be shown. We also see the meaning of weighted geometric means in the gyrogroup structure with finite dimension and find the formulas of weighted geometric means of $2$-by-$2$ positive definite matrices and density matrices.

\vspace{5mm}

\noindent {\bf Keywords}: Positive definite operator, Loewner order, geometric mean, spectral geometric mean, gyrogroup, gyroline
\end{abstract}

\section{Introduction}

A geometric mean of two positive real numbers is the length of the side of the square with the same area of the rectangle with sides of given positive real numbers. It was introduced first in Euclid's \emph{Elements} (Book II, Proposition 14), and since then many characterizations have been studied (see Section 1 in \cite{LL}). During several decades, a variety of construction schemes of the geometric means for positive operators and matrices together with their properties and applications have been widely developed.

Since Pusz and Woronowicz \cite{PW} has defined a geometric mean $A \# B$ of positive definite matrices $A$ and $B$,
\begin{displaymath}
A \# B = A^{1/2} (A^{-1/2} B A^{-1/2})^{1/2} A^{1/2}
\end{displaymath}
many properties on the finite- and infinite-dimensional settings have been found. The geometric mean of positive definite matrices $A$ and $B$ is a unique midpoint of the Riemannian geodesic, called the weighted geometric mean of $A$ and $B$:
\begin{displaymath}
\gamma(t) = A^{1/2} (A^{-1/2} B A^{-1/2})^{t} A^{1/2} =: A \#_{t} B, \ t \in [0,1]
\end{displaymath}
connecting $A$ and $B$ with respect to the Riemannian trace metric $\delta(A, B) = \Vert \log A^{-1/2} B A^{-1/2} \Vert_{2}$. Since Kubo and Ando \cite{KA} have established the geometric mean of positive definite operators, many theoretical and computational research topics including the operator inequality and the extension theory to multi-variable geometric means have been widely studied.

On the other hand, Fiedler and Pt\'{a}k \cite{FP} have introduced and studied a new geometric mean $A \natural B$ of two positive definite matrices $A$ and $B$:
\begin{displaymath}
A \natural B = (A^{-1} \# B)^{1/2} A (A^{-1} \# B)^{1/2}
\end{displaymath}
having analogous properties of the geometric mean. Since the eigenvalues of $A \natural B$ are the same as the positive square roots of the eigenvalues of $A B$, we call it the spectral geometric mean. It can be naturally generalized to the weighted spectral geometric mean as a differentiable curve
\begin{displaymath}
\beta(t) = (A^{-1} \# B)^{t} A (A^{-1} \# B)^{t} =: A \natural_{t} B, \ t \in [0,1].
\end{displaymath}
Note that $A \natural B = A \natural_{1/2} B$ and $A \natural_{t} B = A^{1-t} B^{t}$ for any commuting $A$ and $B$. Several interesting properties of the weighted spectral geometric mean have been found \cite{AKL, KL, LeeL}, but not as much as the geometric mean.

The main goal of this paper is to investigate those weighted geometric means of positive definite operators with operator inequality and applications to the finite-dimensional gyrogroup structure, which provides an algebraic tool for the hyperbolic geometry and the theory of special relativity. In more details, we give a geometric property of the spectral geometric mean for positive definite operators as a metric midpoint, see the meaning of two weighted geometric means on the gyrogroup structure of positive definite Hermitian matrices and density matrices, respectively, and find the explicit formulas of two weighted geometric means for $2 \times 2$ positive definite matrices and density matrices.

The structure of this article is organized as follows. We show in Section 2 the component-wise bijection of these weighted geometric means for positive definite operators and give a geometric property of the spectral geometric mean as a midpoint with respect to the new semi-metric $d(A, B) = 2 \Vert \log (A^{-1} \# B) \Vert$. In Section 3 we prove several operator inequalities of two weighted geometric means including the Ando-Hiai inequality with sufficient condition of the spectral geometric mean. In Section 4 we see a connection between the weighted (spectral) geometric means and the gyroline (cogyroline, respectively) on the open convex cone of positive definite matrices and on the gyrovector space of invertible density matrices. We also provide in Section 5 the explicit formulas of the weighted geometric means of $2 \times 2$ positive definite matrices and invertible density matrices as a linear combination. Finally, we close with some remarks about the semi-metric and an inequality relation with the Riemannian trace metric.

\section{Two-variable geometric means of positive definite operators}

Let $B(\mathcal{H})$ be the Banach space of all bounded linear operators on a Hilbert space $\mathcal{H}$ equipped with the inner product $\langle \cdot, \cdot \rangle$ and the operator norm $\Vert \cdot \Vert$. Let $S(\mathcal{H}) \subset B(\mathcal{H})$ be the closed subspace of all self-adjoint operators, and let $\mathbb{P} \subseteq S(\mathcal{H})$ be the open convex cone of all positive definite operators. Note that $A \in \mathbb{P}$ means that $\langle x, Ax \rangle > 0$ for all nonzero $x \in \mathcal{H}$. For $X, Y \in S(\mathcal{H})$ we denote as $X \leq Y$ if and only if $Y - X$ is positive semi-definite, and as $X < Y$ if and only if $Y - X$ is positive definite: this is known as the Loewner order on $S(\mathcal{H})$. For the case that the dimension of the Hilbert space $\mathcal{H}$ is finite, say $\dim(\mathcal{H}) = n < \infty$, we denote as $\mathbb{P}_{n}$ instead of $\mathbb{P}$.

\subsection{Geometric means}

Pusz and Woronowicz \cite{PW} have first introduced a geometric mean $A \# B$ of positive definite Hermitian matrices $A$ and $B$:
\begin{displaymath}
A \# B = A^{1/2} (A^{-1/2} B A^{-1/2})^{1/2} A^{1/2}.
\end{displaymath}
Since then Kubo and Ando \cite{KA} have established the geometric mean of positive definite operators in $\mathbb{P}$ with several remarkable properties. It coincides with the unique positive definite solution $X \in \mathbb{P}$ of the Riccati equation $X A^{-1} X = B$, and moreover, from \cite{KA}
\begin{equation} \label{E:geomean}
A \# B = \max \left\{ X \in S(\mathcal{H}) :
\left(
  \begin{array}{cc}
    A & X \\
    X & B \\
  \end{array}
\right) \geq 0 \right\}.
\end{equation}
Note that \eqref{E:geomean} is proved by the Schur complement, the Loewner-Heinz inequality in Lemma \ref{L:Loewner-Heinz} and the order preserving of congruence transformation.

The Thompson metric on $\mathbb{P}$ is defined by
$d_{T}(A,B) = \Vert \log(A^{-1/2} B A^{-1/2}) \Vert$, where $\Vert X \Vert$ denotes the operator norm of $X$. It is known that $d_{T}$ is a complete metric on $\mathbb{P}$ and that
\begin{displaymath}
d_{T}(A,B) = \max \{ \log M(B/A), \log M(A/B) \},
\end{displaymath}
where $M(B/A) = \inf \{\alpha > 0: B \leq \alpha A\}$ is the same as the largest eigenvalue of $A^{-1/2} B A^{-1/2}$: see \cite{Th}.
The geometric mean curve
\begin{equation} \label{E:geodesic}
\displaystyle [0,1] \ni t \mapsto A \#_{t} B := A^{1/2} (A^{-1/2} B A^{-1/2})^{t} A^{1/2}.
\end{equation}
is a minimal geodesic from $A$ to $B$ for the Thompson metric. It can be extended to the curve $t \mapsto A \#_{t} B$ for all real numbers $t$. We call $A \#_{t} B$ for $t \in [0,1]$ the \emph{weighted geometric mean} of $A$ and $B$ on $\mathbb{P}$, and especially, its metric midpoint $A \# B := A \#_{1/2} B$.

\begin{theorem} \label{T:bijection}
Let $A, B \in \mathbb{P}$.
\begin{itemize}
\item[(1)] The map $X \mapsto A \#_{t} X$ is a bijection on $\mathbb{P}$ for any $t \in (0,1]$.
\item[(2)] The map $X \mapsto X \#_{t} B$ is a bijection on $\mathbb{P}$ for any $t \in [0,1)$.
\end{itemize}
\end{theorem}

\begin{proof}
We first show (1). Then (2) holds easily from (1), since $X \#_{t} B = B \#_{1-t} X$. Let $t \in (0,1]$.
\begin{itemize}
\item[(i)] For $X, Y \in \mathbb{P}$, assume that $A \#_{t} X = A \#_{t} Y$. Then $(A^{-1/2} X A^{-1/2})^{t} = (A^{-1/2} Y A^{-1/2})^{t}$ by congruence transformation. So $A^{-1/2} X A^{-1/2} = A^{-1/2} Y A^{-1/2}$ by taking the $(1/t)$-power map. Thus, $X = Y$, that is, the map $X \mapsto A \#_{t} X$ is injective.

\item[(ii)] For any $C \in \mathbb{P}$, set $X = A \#_{1/t} C$. Then $X \in \mathbb{P}$ and $A \#_{t} X = A \#_{t} (A \#_{1/t} C) = C$. So the map $X \mapsto A \#_{t} X$ is surjective.
\end{itemize}
\end{proof}

\begin{remark}
On the open convex cone $\mathbb{P}_{n}$ of positive definite Hermitian matrices, first Moakher \cite{Mo} and then Bhatia and Holbrook \cite{BH} suggested the multi-variable geometric mean by taking the mean to be the unique minimizer of the weighted sum of squares of Riemannian trace distances to each variable $A_{1}, \dots, A_{m}$:
\begin{displaymath}
\Lambda(\omega; A_{1}, \dots, A_{m}) = \underset{X \in \mathbb{P}_{n}}{\arg \min} \sum_{i=1}^{m} w_{i} \delta^{2} (X, A_{i})
\end{displaymath}
where $\delta(A, B) = \Vert \log A^{-1/2} B A^{-1/2} \Vert_{2}$ is the Riemannian trace distance and $\omega = (w_{1}, \dots, w_{m})$ is the positive probability vector in $\mathbb{R}^{m}$. This idea in a Riemannian manifold has been anticipated by \'{E}lie Cartan, and later by Karcher \cite{Kar} $\Lambda(\omega; A_{1}, \dots, A_{m})$ coincides with a unique solution $X \in \mathbb{P}_{n}$ of the non-linear equation, named the Karcher equation,
\begin{equation} \label{E:Karcher}
\sum_{i=1}^{m} w_{i} \log (X^{1/2} A_{i}^{-1} X^{1/2}) = 0.
\end{equation}
We call $\Lambda(\omega; A_{1}, \dots, A_{m})$ the Cartan mean, Riemannian mean or Karcher mean.

This theory does not readily carry over to the setting $\mathbb{P}$ of positive definite operators on a Hilbert space, because there is neither such Riemannian structure nor non-positive curvature metric. Nevertheless, it has been shown in \cite{LL14} that the Karcher equation \eqref{E:Karcher} has a unique solution in the infinite-dimensional setting $\mathbb{P}$, so the unique solution can be defined as the Karcher mean $\Lambda(\omega; A_{1}, \dots, A_{m})$. Since $\Lambda(1-t, t; A, B) = A \#_{t} B$ for any $t \in [0,1]$ we have that $A \#_{t} B$ is the unique solution $X \in \mathbb{P}$ of the equation
\begin{displaymath}
(1-t) \log (X^{1/2} A^{-1} X^{1/2}) + t \log (X^{1/2} B^{-1} X^{1/2}) = 0.
\end{displaymath}
\end{remark}

\subsection{Spectral geometric means}

Fiedler and Pt\'{a}k have introduced and studied in \cite{FP} a new geometric mean of two positive definite matrices, called the \emph{spectral geometric mean}, which possesses analogous properties of the geometric mean. The spectral geometric mean of $A$ and $B$, denoted by $A \natural B$, is defined by
\begin{displaymath}
A \natural B := (A^{-1} \# B)^{1/2} A (A^{-1} \# B)^{1/2}.
\end{displaymath}
One of the most important properties is that $(A \natural B)^{2}$ is positively similar to $A B$, and hence, the eigenvalues of $A \natural B$ are the same as the positive square roots of the eigenvalues of $A B$.

H. Lee and Y. Lim \cite{LeeL} have extended the theory of spectral geometric means of positive definite matrices to symmetric cones, and also provided a weighted version of spectral geometric mean with interesting properties. We in this section consider the spectral geometric mean of positive definite operators with its properties. For any $t \in [0,1]$ the \emph{weighted spectral geometric mean} of $A$ and $B$ in $\mathbb{P}$ is defined as a differentiable curve
\begin{equation} \label{E:sp-geomean}
t \in [0,1] \mapsto A \natural_{t} B := (A^{-1} \# B)^{t} A (A^{-1} \# B)^{t}.
\end{equation}
We can easily see that $A \natural_{0} B = A, \ A \natural_{1} B = B$, and $A \natural B = A \natural_{1/2} B$.
We list some properties of the weighted spectral geometric mean (see \cite{LeeL}).

\begin{proposition} \label{P:SG}
Let $A, B \in \mathbb{P}$ and $t \in [0,1]$. Then $A \natural_{t} B$ is a unique positive definite solution $X \in \mathbb{P}$ of the equation
\begin{eqnarray*}
(A^{-1} \# B)^{t} = A^{-1} \# X.
\end{eqnarray*}
\end{proposition}


\begin{lemma} \label{L:SG}
Let $A, B\in \mathbb{P}$ and $s, t, u \in [0,1]$.
\begin{itemize}
\item[(1)] $A \natural_{t} B = A^{1-t} B^{t}$ if $A$ and $B$ commute.
\item[(2)] $(a A) \natural_{t} (b B) = a^{1-t} b^{t} (A \natural_{t} B)$ for any $a, b > 0$.
\item[(3)] $U^{*} (A \natural_{t} B) U = (U^{*} A U) \natural_{t} (U^{*} B U)$ for any unitary operator $U$.
\item[(4)] $A \natural_{t} B = B \natural_{1-t} A$.
\item[(5)] $(A \natural_{s} B) \natural_{t} (A \natural_{u} B) = A \natural_{(1-t)s + t u} B$.
\item[(6)] $(A \natural_{t} B)^{-1} = A^{-1} \natural_{t} B^{-1}$.
\end{itemize}
\end{lemma}

\begin{theorem}
Let $A, B \in \mathbb{P}$.
\begin{itemize}
\item[(1)] The map $X \mapsto A \natural_{t} X$ is a bijection on $\mathbb{P}$ for any $t \in (0,1]$.
\item[(2)] The map $X \mapsto X \natural_{t} B$ is a bijection on $\mathbb{P}$ for any $t \in [0,1)$.
\end{itemize}
\end{theorem}

\begin{proof}
We first show (1). Then (2) holds easily from (1), since $X \natural_{t} B = B \natural_{1-t} X$ by Lemma \ref{L:SG} (4). Let $t \in (0,1]$.
\begin{itemize}
\item[(i)] For $X, Y \in \mathbb{P}$, assume that $A \natural_{t} X = A \natural_{t} Y$. Taking congruence transformation by $A^{1/2}$ yields
\begin{displaymath}
[ A^{1/2} (A^{-1} \# X)^{t} A^{1/2} ]^{2} = A^{1/2} (A \natural_{t} X) A^{1/2} = A^{1/2} (A \natural_{t} Y) A^{1/2} = [ A^{1/2} (A^{-1} \# Y)^{t} A^{1/2} ]^{2}.
\end{displaymath}
Then $A^{1/2} (A^{-1} \# X)^{t} A^{1/2} = A^{1/2} (A^{-1} \# Y)^{t} A^{1/2}$, so $A^{-1} \# X = A^{-1} \# Y$. By Theorem \ref{T:bijection} (1) $X = Y$, that is, the map $X \mapsto A \natural_{t} X$ is injective.

\item[(ii)] For any $C \in \mathbb{P}$, set $X = A \natural_{1/t} C$. Then $X \in \mathbb{P}$ and $A \natural_{t} X = A \natural_{t} (A \natural_{1/t} C) = C$ by Lemma \ref{L:SG} (5). So the map $X \mapsto A \#_{t} X$ is surjective.
\end{itemize}
\end{proof}

Unfortunately, the spectral geometric mean does not satisfy the monotonicity and arithmetic-geometric-harmonic mean inequalities with respect to the Loewner order on $S(\mathcal{H})$.
In \cite[Proposition 2.3]{KL}, on the other hand, an upper bound and a lower bound of the weighted spectral geometric mean are provided.
\begin{proposition} \label{P:bounds}
Let $A, B \in \mathbb{P}$ and $t \in [0,1]$. Then
\begin{displaymath}
\displaystyle 2^{1 + t} (A + B^{-1})^{-t} - A^{-1} \leq A \natural_{t} B \leq [ 2^{1 + t} (A^{-1} + B)^{-t} - A ]^{-1}.
\end{displaymath}
\end{proposition}

No metric property for the spectral geometric mean of positive definite operators and matrices is known. We show that the spectral geometric mean is a midpoint with respect to the following semi-metric on $\mathbb{P}$.
\begin{lemma} \label{L:semimetric}
The map $d: \mathbb{P} \times \mathbb{P} \to [0, \infty)$ given by
\begin{displaymath}
\displaystyle d(A, B) = 2 \Vert \log (A^{-1} \# B) \Vert
\end{displaymath}
is a semi-metric, where $\Vert \cdot \Vert$ denotes the operator norm.
\end{lemma}

\begin{proof}
We need to show that the map $d: \mathbb{P} \times \mathbb{P} \to [0, \infty)$ satisfies all axioms of metric except the triangle inequality.
\begin{itemize}
\item[(1)] Obviously $d(A, B) \geq 0$ for all $A, B \in \mathbb{P}$.

\item[(2)] If $A = B$ then $A^{-1} \# B = I$, and $\log (A^{-1} \# B) = O$. So $d(A, B) = 0$. Conversely, $d(A, B) = 0$ implies that $\log (A^{-1} \# B) = O$, that is, $A^{-1} \# B = I$. Then $(A^{1/2} B A^{1/2})^{1/2} = A$, and $A^{1/2} B A^{1/2} = A^{2}$. Thus, $B = A$.

\item[(3)] By the invariance under inversion and symmetry of the geometric mean
\begin{displaymath}
\begin{split}
d(A, B) & = 2 \Vert \log (A^{-1} \# B) \Vert = 2 \Vert - \log (A^{-1} \# B) \Vert = 2 \Vert \log (A^{-1} \# B)^{-1} \Vert \\
& = 2 \Vert \log (A \# B^{-1}) \Vert = 2 \Vert \log (B^{-1} \# A) \Vert = d(B, A).
\end{split}
\end{displaymath}
\end{itemize}
\end{proof}

We see several interesting properties of the semi-metric $d$.
\begin{proposition} \label{P:semimetric}
Let $A, B \in \mathbb{P}$.
\begin{itemize}
\item[(1)] $d(\alpha A, \alpha B) = d(A, B)$ for any $\alpha > 0$.
\item[(2)] $d(A^{-1}, B^{-1}) = d(A, B)$.
\item[(3)] $d(U A U^{*}, U B U^{*}) = d(A, B)$ for any unitary operator $U$.
\item[(4)] $d(A^{t}, B^{t}) = |t| d(A, B)$ for any commuting operators $A, B$.
\end{itemize}
\end{proposition}

\begin{proof}
The items (1) and (2) follow from the homogeneity and invariance under inversion of the geometric mean $\#$.
\begin{itemize}
\item[(3)] For any unitary operator $U$, $(U A U^{*})^{-1} \# U B U^{*} = U A^{-1} U^{*} \# U B U^{*} = U (A^{-1} \# B) U^{*}$, so
\begin{displaymath}
\log [ (U A U^{*})^{-1} \# U B U^{*} ] = U \log (A^{-1} \# B) U^{*}.
\end{displaymath}
Since the operator norm is unitarily invariant, $d(U A U^{*}, U B U^{*}) = 2 \Vert \log (A^{-1} \# B) \Vert = d(A, B)$.

\item[(4)] Since $A^{-1} \# B = A^{-1/2} B^{1/2}$ for any commuting $A, B$, we have $d(A, B) = \Vert \log A - \log B \Vert$. So
\begin{displaymath}
d(A^{t}, B^{t}) = \Vert \log A^{t} - \log B^{t} \Vert = |t| \Vert \log A - \log B \Vert = |t| d(A,B).
\end{displaymath}
\end{itemize}
\end{proof}

\begin{theorem} \label{T:semimetric}
The spectral geometric mean $A \natural B$ is a midpoint of $A$ and $B$ with respect to the semi-metric $d$ on $\mathbb{P}$.
\end{theorem}

\begin{proof}
By Proposition \ref{P:SG} for $t = 1/2$ we have $(A^{-1} \# B)^{1/2} = A^{-1} \# (A \natural B)$ for positive definite operators $A$ and $B$. Then
\begin{displaymath}
d(A, A \natural B) = 2 \Vert \log (A^{-1} \# (A \natural B)) \Vert = 2 \Vert \log (A^{-1} \# B)^{1/2} \Vert = \Vert \log (A^{-1} \# B) \Vert = \frac{1}{2} d(A, B).
\end{displaymath}
Similarly, by Lemma \ref{L:SG} (4), the above result and the symmetry of semi-metric $d$
\begin{displaymath}
d(B, A \natural B) = d(B, B \natural A) = \frac{1}{2} d(B, A) = \frac{1}{2} d(A, B).
\end{displaymath}
\end{proof}

\begin{remark}
We give some remarks and an open question for the semi-metric.
\begin{itemize}
\item[(1)] Note that the semi-metric $d$ can be considered as a symmetric divergence. Unfortunately, the semi-metric $d$ does not hold the triangle inequality. For instance, let
\begin{displaymath}
A =
\left(
  \begin{array}{cc}
    5 & 0 \\
    0 & 1/5 \\
  \end{array}
\right), B =
\left(
  \begin{array}{cc}
    2 & -3 \\
    -3 & 5 \\
  \end{array}
\right), C =
\left(
  \begin{array}{cc}
    1 & -2 \\
    -2 & 5 \\
  \end{array}
\right).
\end{displaymath}
Then $A, B, C \in \mathbb{P}_{2}$ with determinant $1$, and by using MATLAB
\begin{displaymath}
d(A, B) = 1.117270 \cdots, \ d(B, C) = 0.173732 \cdots, \ d(A, C) = 1.305274 \cdots.
\end{displaymath}
So $d(A, C) > d(A, B) + d(B, C)$.

\item[(2)] Note that for any $t \in [0,1]$
\begin{center}
$d(A, A \natural_{t} B) = t d(A, B)$ \ and \ $d(B, A \natural_{t} B) = (1-t) d(A, B)$.
\end{center}
Indeed, by Proposition \ref{P:SG}
\begin{displaymath}
d(A, A \natural_{t} B) = 2 \Vert \log (A^{-1} \# (A \natural_{t} B)) \Vert = 2 \Vert \log (A^{-1} \# B)^{t} \Vert = 2t \Vert \log (A^{-1} \# B) \Vert = t d(A, B).
\end{displaymath}
Similarly, $d(B, A \natural_{t} B) = d(B, B \natural_{1-t} A) = (1-t) d(B, A) = (1-t) d(A, B)$ by Lemma \ref{L:SG} (4). It gives an affirmative answer that $A \natural_{t} B$ is a geodesic for the semi-metric $d$, but it is a challengeable problem. That is, for any $s, t \in [0,1]$
\begin{displaymath}
d(A \natural_{s} B, A \natural_{t} B) = |s-t| d(A, B).
\end{displaymath}
\end{itemize}
\end{remark}

\section{Operator inequalities}

In this section we see several operator inequalities of weighted geometric means. The following are crucial to prove the main results.

\begin{lemma} [The Loewner-Heinz inequality] \label{L:Loewner-Heinz}
Let $C \in S(\mathcal{H})$ and let $A, B \geq 0$. If $C^{2} \leq A \leq B$, then $C \leq A^{1/2} \leq B^{1/2}$.
\end{lemma}

Lemma \ref{L:Loewner-Heinz} has been proved for a Hermitian matrix $C$ and positive semi-definite matrices $A, B$ in \cite{LL}, but it also holds for operators. Moreover, it yields that for positive semi-definite operators $A$ and $B$,
\begin{center}
$A \leq B$ \ implies \ $A^{t} \leq B^{t}$ for any $t \in [0,1]$.
\end{center}
This is also known as the Loewner-Heinz inequality, which is the most used in the fields of operator inequalities and operator means.

The following are well-known results of the Loewner partial order applied to positive operators.

\begin{lemma} \label{L:order}
Let $X, Y \in S(\mathcal{H})$, $S \in B(\mathcal{H})$, and $t \in [0,1]$.
\begin{itemize}
\item[(1)] $0 \leq X \leq Y$ implies $0 \leq S X S^{*} \leq S Y S^{*}$.
\item[(2)] $0 < X \leq Y$ if and only if $0 < Y^{-1} \leq X^{-1}$.
\end{itemize}
\end{lemma}

The following has been proved for positive definite Hermitian matrices in \cite[Corollary 4.4.5]{Bh}, but it is also valid for positive definite operators.
\begin{lemma} \label{L:Furuta}
Let $0 \leq B \leq A$. Then for any $p >0$
\begin{displaymath}
A^{p} \# B^{-p} \geq I.
\end{displaymath}
\end{lemma}

\begin{theorem}
Let $A, B \in \mathbb{P}$ with $A \# B \leq I$. Then
\begin{displaymath}
A^{p+1} \# (A \#_{\frac{p}{2}} B) \leq A
\end{displaymath}
for any $p > 0$. Especially, $A^{3} \# B \leq A$ for $p = 2$.
\end{theorem}

\begin{proof}
Assume that $A \# B \leq I$ for $A, B \in \mathbb{P}$. Taking the congruence transformation by $A^{-1/2}$ on both sides implies $(A^{-1/2} B A^{-1/2})^{1/2} \leq A^{-1}$ by Lemma \ref{L:order} (1). By applying Lemma \ref{L:Furuta} we have
\begin{displaymath}
A^{-p} \# (A^{-1/2} B A^{-1/2})^{-p/2} \geq I
\end{displaymath}
for any $p > 0$. Taking inverse on both sides and congruence transformation by $A^{1/2}$ yield
\begin{displaymath}
A^{p+1} \# (A \#_{\frac{p}{2}} B) = A^{p+1} \# A^{1/2} (A^{-1/2} B A^{-1/2})^{p/2} A^{1/2} \leq A
\end{displaymath}
from Lemma \ref{L:order}.
\end{proof}

\begin{theorem}
Let $A, B \in \mathbb{P}$. Then the following are equivalent.
\begin{itemize}
\item[(1)] $A^{-1} \natural B \leq I$;
\item[(2)] $A \natural B^{-1} \geq I$;
\item[(3)] $A \# B \leq A$;
\item[(4)] $A \# B \geq B$;
\item[(5)] $B \leq A$.
\end{itemize}
\end{theorem}

\begin{proof}
By Lemma \ref{L:SG} (6), items (1) and (2) are equivalent. Moreover, items (3) and (4) are equivalent from \cite[Theorem 4.2]{FP}

We now show (1) $\Longrightarrow$ (3) $\Longrightarrow$ (5) $\Longrightarrow$ (1).
\begin{itemize}
\item[(1)] $\Longrightarrow$ (3) Assume that $A^{-1} \natural B \leq I$, that is,
\begin{displaymath}
\displaystyle (A \# B)^{1/2} A^{-1} (A \# B)^{1/2} \leq I.
\end{displaymath}
By Lemma \ref{L:order} (1) we have $A^{-1} \leq (A \# B)^{-1}$, equivalently $A \# B \leq A$ by Lemma \ref{L:order} (2).

\item[(3)] $\Longrightarrow$ (5) Assume that $A \# B \leq A$. Using Lemma \ref{L:order} (1) again, we have
\begin{displaymath}
\displaystyle (A^{-1/2} B A^{-1/2})^{1/2} \leq I.
\end{displaymath}
This implies $A^{-1/2} B A^{-1/2} \leq I$, and so $B \leq A$.

\item[(5)] $\Longrightarrow$ (1) This follows from the opposite steps shown as (1) $\Longrightarrow$ (3) $\Longrightarrow$ (5).
\end{itemize}
\end{proof}

\begin{lemma} \label{L:contraction}
Let $S \in S(\mathcal{H})$ and let $X \in \mathbb{P}$. Then $S X S \leq X$ implies $S \leq I$.
\end{lemma}

\begin{proof}
Assume that $S X S \leq X$. Then by Lemma \ref{L:order} (1)
\begin{displaymath}
\displaystyle (X^{1/2} S X^{1/2})^{2} = X^{1/2} (S X S) X^{1/2} \leq X^{2}.
\end{displaymath}
By Lemma \ref{L:Loewner-Heinz}, $X^{1/2} S X^{1/2} \leq X$. Using Lemma \ref{L:order} (1) again, we obtain $S \leq I$.
\end{proof}

\begin{remark}
In general, the converse of Lemma \ref{L:contraction} does not hold. For instance, let
\begin{displaymath}
S =
\left(
  \begin{array}{cc}
    0 & 1 \\
    1 & 0 \\
  \end{array}
\right), \
X =
\left(
  \begin{array}{cc}
    a & b \\
    b & c \\
  \end{array}
\right) \in \mathbb{P}_{2},
\end{displaymath}
where $a, c > 0$, $a \neq c$ and $ac - b^{2} > 0$. Note that $S$ is a permutation matrix, so $S^{2} = I$. That is, $S$ is a contraction. However,
\begin{displaymath}
X - S X S =
\left(
  \begin{array}{cc}
    a & b \\
    b & c \\
  \end{array}
\right) -
\left(
  \begin{array}{cc}
    c & b \\
    b & a \\
  \end{array}
\right) =
\left(
  \begin{array}{cc}
    a-c & 0 \\
    0 & c-a \\
  \end{array}
\right),
\end{displaymath}
which is not positive semi-definite, in general.
\end{remark}

Ando and Hiai \cite{AH} have established a characterization of two-variable geometric means such as
\begin{center}
$A \# B \leq I$ \ implies \ $A^{p} \# B^{p} \leq I$
\end{center}
for all $p \geq 1$ and positive definite matrices $A$ and $B$. We call it the \emph{Ando-Hiai inequality}. Yamazaki \cite{Ya} generalized the Ando-Hiai inequality to the Riemannian mean (or Cartan mean, Karcher mean) of positive definite matrices, which is the multi-variable geometric mean defined as the unique minimizer of the weighted sum of squares of the Riemannian trace distances to each variable. Lawson and Lim \cite{LL14} naturally defined the Karcher mean of positive definite operators as the unique positive definite solution of the Karcher equation, and extended Yamazaki's result to positive definite operators.

We give a different sufficient condition for the Ando-Hiai inequality in terms of the spectral geometric mean.
\begin{theorem} \label{T:main}
Let $A, B \in \mathbb{P}$ and $t \in (0,1]$. Then $A^{-1} \natural_{t} B \leq A^{-1}$ implies $A^{p} \# B^{p} \leq I$ for any $p \geq 1$.
\end{theorem}

\begin{proof}
Assume that $A^{-1} \natural_{t} B = (A \# B)^{t} A^{-1} (A \# B)^{t} \leq A^{-1}$. By Lemma \ref{L:contraction} we have $(A \# B)^{t} \leq I$, equivalently $A \# B \leq I$. By the Ando-Hiai inequality we obtain $A^{p} \# B^{p} \leq I$ for any $p \geq 1$.
\end{proof}

\begin{remark}
The following are sufficient conditions of the inequality $A \# B \leq I$ for $A, B \in \mathbb{P}$:
\begin{itemize}
\item[(1)] $A \leq I$ and $B \leq I$;
\item[(2)] $\log A + \log B \leq 0$;
\item[(3)] $A^{-1} \natural_{t} B \leq A^{-1}$ for any $t \in (0,1]$.
\end{itemize}
Obviously, (1) implies (2) since the logarithmic map is operator monotone, but the converse is not true in general. On the other hand, the relations between (1) and (3), or between (2) and (3) do not know yet. So (3) is a new sufficient condition of the inequality $A \# B \leq I$, and also of the Ando-Hiai inequality: $A^{p} \# B^{p} \leq I$ for any $p \geq 1$.
\end{remark}

\section{Gyrogroups and gyrolines}

We review in this section the algebraic structure of a gyrogroup as a natural extension of a group into the regime of the nonassociative algebra. We then introduce a gyrovector space providing the setting for hyperbolic geometry just as a vector space provides the setting for Euclidean geometry. A. A. Ungar has introduced and intensely studied the theory of gyrogroups and gyrovector spaces in a series of papers and books; see \cite{Un08} and its bibliography. We also see what geometric means of positive definite Hermitian matrices are in the sense of gyrovector spaces.

\subsection{Gyrogroups and gyrovector spaces}

The binary operation in a gyrogroup is not associative, in general. The breakdown of associativity for gyrogroup operations is salvaged in a modified form, called gyroassociativity. The axioms for a (gyrocommutative) gyrogroup $G$ is reminiscent of those for a (commutative) group.
\begin{definition} \label{D:gyrogroup}
A binary system $(G, \oplus)$ is a \emph{gyrogroup} if it satisfies
the following axioms for all $a, b, c \in G$:
\begin{itemize}
\item[(G1)] $e \oplus a = a \oplus e = a$ \ (existence of identity);
\item[(G2)] $a \oplus (\ominus a) = (\ominus a) \oplus a = e$ \ (existence of inverses);
\item[(G3)] There is an automorphism $\gyr[a,b] : G \to G$ for each $a, b \in G$ such that
 \begin{center}
 $a \oplus (b \oplus c) = (a \oplus b) \oplus \gyr[a,b]c$ \ (gyroassociativity);
 \end{center}
\item[(G4)] $\gyr[e, a] =$ id$_{G}$, where id$_{G}$ is the identity map on $G$;
\item[(G5)] $\gyr[a \oplus b, b] = \gyr[a,b]$ \ (loop property).
\end{itemize}
A gyrogroup $(G, \oplus)$ is \emph{gyrocommutative} if it satisfies
\begin{center}
$a \oplus b = \gyr[a,b](b \oplus a)$ \ (gyrocommutativity).
\end{center}
A gyrogroup is \emph{uniquely $2$-divisible} if for every $b \in G$,
there exists a unique element $a \in G$ such that $a \oplus a = b$.
\end{definition}

In (G3) the automorphism $\gyr[a,b]$ for each $a,b \in G$ is called the \emph{Thomas gyration} or the \emph{gyroautomorphism}, simply \emph{gyration}. From (G2) and (G3) we have
\begin{displaymath}
\displaystyle \gyr[a,b]c = \ominus (a \oplus b) \oplus [ a \oplus (b \oplus c) ]
\end{displaymath}
for all $a, b, c \in G$. In Euclidean space it plays a role of rotation in the plane spanned by $\{a, b\}$ leaving the orthogonal complement fixed.

\begin{definition} \label{D:cooperation}
Let $(G, \oplus)$ be a gyrogroup. The \emph{gyrogroup cooperation} is a binary operation in $G$ given by
\begin{displaymath}
a \boxplus b = a \oplus \gyr[a, \ominus b] b
\end{displaymath}
for all $a, b \in G$. The groupoid $(G, \boxplus)$ is said to be the cogyrogroup associated with the gyrogroup $(G, \oplus)$.
\end{definition}

The gyrogroup cooperation $\boxplus$ gives a useful criterion for the gyrogroup to be gyrocommutative.
\begin{theorem} \cite[Theorem 3.2, Theorem 3.4]{Un08}
Let $(G, \oplus)$ be a gyrogroup. The following are equivalent.
\begin{itemize}
\item[(1)] $G$ is gyrocommutative.
\item[(2)] $\ominus (a \oplus b) = \ominus a \ominus b$ for all $a, b \in G$.
\item[(3)] $a \boxplus b = b \boxplus a$ for all $a, b \in G$.
\end{itemize}
\end{theorem}

In the same way that vector spaces are commutative groups of vectors admitted scalar multiplication, gyrovector spaces are gyrocommutative gyrogroups of gyrovectors admitted properly scalar multiplication. We give a definition of gyrovector spaces slightly different from Definition 6.2 in \cite{Un08}.
\begin{definition} \label{D:gyrovector space}
A gyrocommutative gyrogroup $(G, \oplus)$ equipped with a scalar multiplication
$$ (t, x) \mapsto t \otimes x: \mathbb{R} \times G \to G $$
is called a \emph{gyrovector space} if it satisfies the following for $s, t \in \mathbb{R}$ and $a, b, c \in G$.
\begin{enumerate}
\item[(V1)] $1 \otimes a = a$, $0 \otimes a = t \otimes e = e$, and $(-1) \otimes a = \ominus a$.
\item[(V2)] $(s + t) \otimes a = s \otimes a \oplus t \otimes a$.
\item[(V3)] $(s t) \otimes a = s \otimes (t \otimes a)$.
\item[(V4)] $\gyr[a,b] (t \otimes c) = t \otimes \gyr[a,b] c$.
\end{enumerate}
\end{definition}

We now see some typical examples of the gyrovector space in the Euclidean space $\mathbb{R}^{n}$.
We consider elements in $\mathbb{R}^{n}$ naturally as column vectors, so that $\uu^{T} \vv$ is the usual inner product of $\uu, \vv \in \mathbb{R}^{n}$ written in matrix form. A. Ungar has introduced in \cite{Un08} two distinctive examples of gyrovector spaces in the open unit ball $\mathbf{B}$ of the $n$-dimensional vector space $\mathbb{R}^{n}$, also corresponding to two models of hyperbolic geometry.

\begin{example}
We define the binary operations $\oplus_{E}$ and $\oplus_{M}$ in $\mathbf{B}$ by
\begin{eqnarray}
\displaystyle \uu \oplus_{E} \vv & = & \frac{1}{1 + \uu^{T} \vv}
\left\{ \uu + \frac{1}{\gamma_{\uu}} \vv + \frac{\gamma_{\uu}}{1 + \gamma_{\uu}} (\uu^{T} \vv) \uu \right\}, \\
\displaystyle \uu \oplus_{M} \vv & = & \frac{1}{1 + 2 \uu^{T} \vv + \Vert \uu \Vert^{2} \Vert \vv \Vert^{2}}
\left\{ \left( 1 + 2 \uu^{T} \vv + \Vert \vv \Vert^{2} \right) \uu + \left( 1 - \Vert \uu \Vert^{2} \right) \vv \right\}
\end{eqnarray}
for any $\uu, \vv \in \mathbf{B}$, where $\gamma_{\mathbf{v}} = \frac{1}{\sqrt{1 - \Vert \mathbf{v} \Vert^2}}$ is the well-known Lorentz (gamma) factor.
The binary systems $(\mathbf{B}, \oplus_{E})$ and $(\mathbf{B}, \oplus_{M})$ form (uniquely 2-divisible) gyrocommutative gyrogroups called the \emph{Einstein gyrogroup} and \emph{M\"obius gyrogroup}, respectively.

We define a map $\otimes: \mathbb{R} \times \mathbf{B} \to \mathbf{B}$ by
\begin{equation} \label{Eq:scalar}
\begin{split}
\displaystyle t \otimes \vv & = \frac{\left( 1 + \Vert \vv \Vert \right)^{t} - \left( 1 - \Vert \vv \Vert \right)^{t}}{\left( 1 + \Vert \vv \Vert \right)^{t} + \left( 1 - \Vert \vv \Vert \right)^{t}} \frac{\vv}{\Vert \vv \Vert} \\
& = \tanh \left( t \tanh^{-1} \Vert \mathbf{v} \Vert \right) \frac{\vv}{\Vert \vv \Vert},
\end{split}
\end{equation}
for $t \in \mathbb{R}$ and $\vv (\neq \mathbf{0}) \in \mathbf{B}$, and define $t \otimes \mathbf{0} := \mathbf{0}$. We call $(\mathbf{B}, \oplus_{E}, \otimes)$ and $(\mathbf{B}, \oplus_{M}, \otimes)$ the \emph{Einstein gyrovector space} and the \emph{M\"obius gyrovector space}, respectively.
\end{example}

The Beltrami-Klein ball model of hyperbolic geometry is algebraically regulated by Einstein gyrovector spaces. The geodesics of this model, called gyrolines, are Euclidean straight lines in the open unit ball. On the other hand, the Poincar\'e ball model of hyperbolic geometry is algebraically regulated by M\"obius gyrovector spaces. The geodesics of this model are Euclidean circular arcs in the open unit ball that intersect the boundary of the ball orthogonally.

\begin{example} \cite[Example 2.2, Example 3.2]{Kim2} \label{Ex:cone}
We define the binary operation $\oplus$ and a scalar multiplication $\circ$ on the open convex cone $\mathbb{P}_{n}$ of positive definite matrices by
\begin{displaymath}
\begin{split}
\displaystyle \oplus: & \ \mathbb{P}_{n} \times \mathbb{P}_{n} \to \mathbb{P}_{n}, \ A \oplus B = A^{1/2} B A^{1/2}, \\
\displaystyle \circ: & \ \mathbb{R} \times \mathbb{P}_{n} \to \mathbb{P}_{n}, \ t \circ A = A^{t}
\end{split}
\end{displaymath}
for any $A, B \in \mathbb{P}_{n}$ and $t \in \mathbb{R}$. Then the system $(\mathbb{P}_{n}, \oplus, \circ)$ forms a gyrovector space, and the gyroautomorphism generated by $A$ and $B$ is given by
\begin{equation} \label{E:gyration}
\displaystyle \gyr[A,B] X = U(A, B) X U(A, B)^{-1},
\end{equation}
where $U(A, B) = (A^{1/2} B A^{1/2})^{-1/2} A^{1/2} B^{1/2}$ is a unitary part of the polar decomposition for $A^{1/2} B^{1/2}$ such that
\begin{displaymath}
\displaystyle A^{1/2} B^{1/2} = (A \oplus B)^{1/2} U(A, B).
\end{displaymath}
\end{example}

\begin{remark}
The inner product on M$_{n}(\mathbb{C})$, the vector space of all $n \times n$ matrices with complex entries, is naturally defined as $\langle A, B \rangle = \tr (A B^{*})$. The gyroautomorphism on $\mathbb{P}_{n}$ preserves the inner product, and so the norm induced by inner product. Indeed, for any $A, B, X, Y \in \mathbb{P}_{n}$
\begin{displaymath}
\begin{split}
\langle \gyr[A,B]X, \gyr[A,B]Y \rangle & = \tr [ U(A, B) X U(A, B)^{-1} (U(A, B) Y U(A, B)^{-1})^{\dag} ] \\
& = \tr [ U(A, B) X Y^{\dag} U(A, B)^{\dag} ] \\
& = \tr [ X Y^{\dag} ] = \langle X, Y \rangle.
\end{split}
\end{displaymath}
\end{remark}

While A. Ungar has explained a gyrogroup structure for $2 \times 2$ density matrices in Chapter 9, \cite{Un08}, we now see an example of gyrovector space for arbitrary dimensional density matrices.

\begin{example} \cite{Kim} \label{Ex:density}
Let $\mathbb{D}_{n}$ be a set of all $n \times n$ invertible density matrices, which are positive definite Hermitian matrices of trace $1$.
We define a binary operation $\odot$ and a scalar multiplication $\star$ given by
\begin{displaymath}
\begin{split}
\displaystyle \odot: & \ \mathbb{D}_{n} \times \mathbb{D}_{n} \to \mathbb{D}_{n}, \ \rho \odot \sigma = \frac{\rho^{\frac{1}{2}} \sigma \rho^{\frac{1}{2}}}{\tr(\rho \sigma)} = \frac{\rho \oplus \sigma}{\tr(\rho \oplus \sigma)} \\
\displaystyle \star: & \ \mathbb{R} \times \mathbb{D}_{n} \to \mathbb{D}_{n}, \ t \ast \rho = \frac{\rho^{t}}{\tr (\rho^{t})} = \frac{t \circ \rho}{\tr (t \circ \rho)}
\end{split}
\end{displaymath}
for any $\rho, \sigma \in \mathbb{D}_{n}$ and $t \in \mathbb{R}$. Then $(\mathbb{D}_{n}, \odot, \star)$ is a gyrovector space. Note that the identity element in $(\mathbb{D}_{n}, \odot, \star)$ is $\displaystyle \frac{1}{n} I_{n}$ and the inverse of $\rho$ is $\displaystyle (-1) \star \rho = \frac{1}{\tr (\rho^{-1})} \rho^{-1}$, where $I_{n}$ denotes the $n \times n$ identity matrix.
\end{example}

\subsection{Gyrolines and cogyrolines}

The \emph{gyroline} passing through the points $a$ and $b$ in the gyrovector space $(G, \oplus, \otimes)$ is defined by
\begin{equation} \label{E:gyroline}
L: \mathbb{R} \times G \times G \to G, \ L(t; a, b) = a \oplus t \otimes (\ominus a \oplus b).
\end{equation}
The gyroline is uniquely determined by given points, and a left gyrotranslation of a gyroline is again a gyroline by Theorem 6.21 in \cite{Un08}. In other words,
\begin{displaymath}
x \oplus L(t; a, b) = L(t; x \oplus a, x \oplus b)
\end{displaymath}
for any $x \in G$. The point $\displaystyle L(1/2; a, b) = a \oplus \frac{1}{2} \otimes (\ominus a \oplus b)$ is called the \emph{gyromidpoint} of given two points $a$ and $b$ in the gyrovector space $(G, \oplus, \otimes)$. By using the gyrogroup cooperation, we can write alternatively
\begin{equation} \label{E:gyromidpoint}
L \left( \frac{1}{2}; a, b \right) = \frac{1}{2} \otimes (a \boxplus b).
\end{equation}

On the other hand, the \emph{cogyroline} passing through the points $a$ and $b$ in the gyrovector space $(G, \oplus, \otimes)$ is defined by
\begin{equation} \label{E:gyroline}
L^{c}: \mathbb{R} \times G \times G \to G, \ L^{c}(t; a, b) = t \otimes (\ominus a \boxplus b) \oplus a.
\end{equation}
The gyroline is uniquely determined by given points, and the point $\displaystyle L^{c}(1/2; a, b) = \frac{1}{2} \otimes (\ominus a \boxplus b) \oplus a$ is called the \emph{cogyromidpoint} of given two points $a$ and $b$ in the gyrovector space $(G, \oplus, \otimes)$.

Gyrolines and cogyrolines play an important role in the hyperbolic analytic geometry regulated by the gyrovector space. See Chapter 6 in \cite{Un08} for more information. We here see the connection of the weighted (spectral) geometric mean and the gyroline (cogyroline, respectively) on the open convex cone $\mathbb{P}$ of positive definite matrices and on the gyrovector space $\mathbb{D}_{n}$ of invertible density matrices.

\begin{theorem} \label{T:gyrolines1}
For any $A, B \in (\mathbb{P}_{n}, \oplus, \circ)$ and $t \in [0,1]$
\begin{displaymath}
\displaystyle L(t; A, B) = A \#_{t} B, \ \ \ L^{c}(t; A, B) = A \natural_{t} B.
\end{displaymath}
\end{theorem}

\begin{proof}
From Example \ref{Ex:cone} we can easily obtain the gyroline on $(\mathbb{P}_{n}, \oplus, \circ)$ passing through $A$ and $B$ such as
\begin{displaymath}
L(t; A, B) = A \oplus t \circ ((-1) \circ A \oplus B) = A^{1/2} (A^{-1/2} B A^{-1/2})^{t} A^{1/2} = A \#_{t} B
\end{displaymath}
for $t \in [0,1]$. Moreover, since
\begin{displaymath}
\begin{split}
(-1) \circ A \boxplus B & = A^{-1} \oplus \gyr[A^{-1}, B^{-1}] B \\
& = A^{-1/2} (A^{1/2} B A^{1/2})^{1/2} A^{-1} (A^{1/2} B A^{1/2})^{1/2} A^{-1/2} \\
& = (A^{-1} \# B)^{2},
\end{split}
\end{displaymath}
the cogyroline passing through $A$ and $B$ is
\begin{displaymath}
L^{c}(t; A, B) = t \circ ((-1) \circ A \boxplus B) \oplus A = (A^{-1} \# B)^{t} A (A^{-1} \# B)^{t} = A \natural_{t} B.
\end{displaymath}
\end{proof}

\begin{theorem} \label{T:gyrolines2}
For any $\rho, \sigma \in (\mathbb{D}_{n}, \odot, \star)$ and $t \in [0,1]$
\begin{displaymath}
\displaystyle L(t; \rho, \sigma) = \frac{1}{\tr (\rho \#_{t} \sigma)} \rho \#_{t} \sigma, \ \ \ L^{c}(t; \rho, \sigma) = \frac{1}{\tr (\rho \natural_{t} \sigma)} \rho \natural_{t} \sigma.
\end{displaymath}
\end{theorem}

\begin{proof}
Let $\rho, \sigma \in (\mathbb{D}_{n}, \odot, \star)$ and $t \in [0,1]$. The formula of $L(t; \rho, \sigma)$ has been shown in \cite[Theorem 4.2]{Kim16}. Note that
\begin{displaymath}
\begin{split}
(-1) \star \rho \boxdot \sigma & = (-1) \star \rho \odot \gyr[(-1) \star \rho, (-1) \star \sigma] \sigma = (-1) \star \rho \odot \gyr[\rho, \sigma] \sigma \\
& = \frac{(-1) \circ \rho \boxplus \sigma}{\tr \left( (-1) \circ \rho \boxplus \sigma \right)}.
\end{split}
\end{displaymath}
The first equality follows from Definition \ref{D:cooperation} and the second follows from Theorem 2.34 in \cite{Un08}.
Thus, by the formula of $L^{c}(t; A, B)$ in Theorem \ref{T:gyrolines1} we obtain
\begin{displaymath}
\displaystyle L^{c}(t; \rho, \sigma) = t \star [(-1) \star \rho \boxdot \sigma] \odot \rho = \frac{t \circ [(-1) \circ \rho \boxplus \sigma] \oplus \rho}{\tr (t \circ [(-1) \circ \rho \boxplus \sigma] \oplus \rho)} = \frac{\rho \natural_{t} \sigma}{\tr (\rho \natural_{t} \sigma)}.
\end{displaymath}
\end{proof}

\section{Special cases: $2 \times 2$ positive definite matrices}

We compute in this section the weighted geometric means of $2 \times 2$ positive definite matrices as the special cases, but important examples in the theoretical and applied areas.

\subsection{$2 \times 2$ positive definite matrices with determinant one}

For any $A, B \in \mathbb{P}_{2}$ with determinant $1$, we can write the geometric mean of $A$ and $B$ to linear combination of $A$ and $B$. For this goal we introduce a map $L_{f}: (0, \infty) \to \mathbb{R}$ constructed by a differentiable function $f: (0, \infty) \to \mathbb{R}$:
\begin{displaymath}
L_{f}(x) =
\left\{
  \begin{array}{ll}
    \displaystyle \frac{f(x) - f(x^{-1})}{x - x^{-1}}, & \hbox{$x \neq 1$;} \\
    f'(1), & \hbox{$x = 1$.}
  \end{array}
\right.
\end{displaymath}
For $f(x) = x^{t}$ we simply write $L_{f}(x)$ as $L_{t}(x)$.

\begin{lemma} \label{L:2-mean}
For any $A, B \in \mathbb{P}_{2}$ with determinant $1$,
\begin{displaymath}
A \#_{t} B = L_{1-t}(\lambda) A + L_{t}(\lambda) B, \ t \in [0,1]
\end{displaymath}
where $\lambda$ is an eigenvalue of $A B^{-1}$.
\end{lemma}

\begin{proof}
Assume that $A \neq B$ for $A, B \in \mathbb{P}_{2}$ with determinant $1$. It has been shown from \cite{Lim13} that there exists a unique minimal geodesic connecting $A$ and $B$ for the Thompson metric $d_{T}$ on $\mathbb{P}_{2}$. Note from \cite{CPR, LL07} that $t \mapsto A \#_{t} B$ is a minimal geodesic for the Thompson metric, and from \cite{Nuss} that $t \mapsto L_{1-t}(\lambda) A + L_{t}(\lambda) B$ is a minimal geodesic for the Thompson metric. By uniqueness, we obtain the desired identity.
\end{proof}

\begin{remark}
It is known from \cite[Proposition 4.1.12]{Bh} that
\begin{displaymath}
A \# B = \frac{A + B}{\sqrt{\det (A + B)}}
\end{displaymath}
for any $A, B \in \mathbb{P}_{2}$ with determinant $1$. So by Lemma \ref{L:2-mean} we have
\begin{displaymath}
L_{\frac{1}{2}}(\lambda) = \frac{1}{\sqrt{\det (A + B)}},
\end{displaymath}
where $\lambda$ is an eigenvalue of $A B^{-1}$.
\end{remark}

\begin{lemma} \cite{An} \label{L:norm}
For any $A, B, X \in B(\mathcal{H})$,
\begin{displaymath}
\left(
  \begin{array}{cc}
    A & X \\
    X^{*} & B \\
  \end{array}
\right) \geq 0,
\end{displaymath}
implies $\Vert B \Vert \leq \sqrt{\Vert A \Vert \cdot \Vert C \Vert}$ for operator norm $\Vert \cdot \Vert$.
\end{lemma}

Applying Lemma \ref{L:norm} to $\mathbb{P}_{2}$ we obtain some interesting inequalities.

\begin{corollary} \label{C:norm}
Let $A, B \in \mathbb{P}_{2}$ with determinants $1$. Then
\begin{displaymath}
\Vert A + B \Vert \leq \sqrt{ \det (A + B) \Vert A \Vert \cdot \Vert B \Vert }.
\end{displaymath}
\end{corollary}

\begin{proof}
Since
\begin{displaymath}
\left(
  \begin{array}{cc}
    A & A \# B \\
    A \# B & B \\
  \end{array}
\right) \geq 0,
\end{displaymath}
Lemma \ref{L:norm} together with Proposition 4.1.12 in \cite{Bh} yield the desired inequality.
\end{proof}

\begin{remark}
Let $A, B \in \mathbb{P}_{2}$ with determinants $\det A = \alpha^{2}$ and $\det B = \beta^{2}$, respectively. Then replacing $A$ and $B$ by $\alpha^{-1} A$ and $\beta^{-1} B$ in Corollary \ref{C:norm} implies
\begin{displaymath}
\sqrt{\alpha \beta} \Vert \alpha^{-1} A + \beta^{-1} B \Vert \leq \left[ \det (\alpha^{-1} A + \beta^{-1} B) \Vert A \Vert \cdot \Vert B \Vert \right]^{1/2}.
\end{displaymath}
\end{remark}

We give a simple formula for the spectral geometric mean of $2 \times 2$ positive definite matrices. The following can be derived from the Cayley-Hamilton theorem for any $2 \times 2$ matrix.

\begin{lemma} \label{L:Cay-Ham}
For any $2 \times 2$ matrix $X$ and any $c \in \mathbb{R}$,
\begin{displaymath}
\det (c I + X) = c^{2} + c \tr (X) + \det (X).
\end{displaymath}
\end{lemma}

\begin{proposition} \label{P:2x2 sp-geomean}
Let $A, B \in \mathbb{P}_{2}$ with determinants $1$. Then
\begin{displaymath}
\displaystyle A \natural_{t} B = \frac{(A^{-1} + B)^{t} A (A^{-1} + B)^{t}}{[ 2 + \tr (A B) ]^{t}}.
\end{displaymath}
\end{proposition}

\begin{proof}
It has been proved in \cite[Proposition 4.1.12]{Bh} that
\begin{displaymath}
\displaystyle A \# B = \frac{A + B}{\sqrt{\det (A + B)}}.
\end{displaymath}
Since $\det(A^{-1}) = 1$, we obtain from the equation \eqref{E:sp-geomean}
\begin{displaymath}
\displaystyle A \natural_{t} B = \left[ \frac{A^{-1} + B}{\sqrt{\det(A^{-1} + B)}} \right]^{t} A \left[ \frac{A^{-1} + B}{\sqrt{\det(A^{-1} + B)}} \right]^{t} \\
= \frac{(A^{-1} + B)^{t} A (A^{-1} + B)^{t}}{\det(A^{-1} + B)^{t}}.
\end{displaymath}
Since $\det(A^{-1} + B) = \det(A^{-1}(I + A B)) = \det(I + A B) = 2 + \tr (A B)$ by Lemma \ref{L:Cay-Ham}, the proof is complete.
\end{proof}

\begin{corollary}
Let $A, B \in \mathbb{P}_{2}$ such that $\det(A) = \alpha^{2}$ and $\det(B) = \beta^{2}$ for some $\alpha, \beta > 0$. Then
\begin{displaymath}
\displaystyle A \natural_{t} B = \left[ \frac{\alpha \beta}{(\alpha \beta)^{2} + 1 + \alpha \beta \tr (A B)} \right]^{t} (\alpha \beta A^{-1} + B)^{t} A (\alpha \beta A^{-1} + B)^{t}.
\end{displaymath}
\end{corollary}

\begin{proof}
By Proposition \ref{P:2x2 sp-geomean} and Lemma \ref{L:Cay-Ham}
\begin{displaymath}
\begin{split}
\displaystyle \left( \frac{A}{\alpha} \right) \natural_{t} \left( \frac{B}{\beta} \right)
& = \frac{(\alpha A^{-1} + \beta^{-1} B)^{t} \alpha^{-1} A (\alpha A^{-1} + \beta^{-1} B)^{t}}{\det(I + (\alpha \beta)^{-1} A B)^{t}} \\
& = \frac{\alpha^{2t-1} (\alpha \beta A^{-1} + B)^{t} A (\alpha \beta A^{-1} + B)^{t}}{\det(\alpha \beta I + A B)^{t}} \\
& = \frac{\alpha^{2t-1} (\alpha \beta A^{-1} + B)^{t} A (\alpha \beta A^{-1} + B)^{t}}{[ (\alpha \beta)^{2} + 1 + \alpha \beta \tr (A B)]^{t}}
\end{split}
\end{displaymath}
Since $\displaystyle \left( \frac{A}{\alpha} \right) \natural_{t} \left( \frac{B}{\beta} \right) = \alpha^{t-1} \beta^{-t} A \natural_{t} B$ by Lemma \ref{L:SG} (2), it leads our result.
\end{proof}

\subsection{$2 \times 2$ invertible density matrices}

In quantum information theory, a physical state can be described as a density matrix n a complex Hilbert space, which is a positive semi-definite Hermitian matrix with trace $1$. In particular, a qubit state is the $2 \times 2$ density matrix whose form is given by
\begin{displaymath}
\rho_{\mathbf{v}} = \frac{1}{2} \left(
\begin{array}{cc}
  1 + v_{3} & v_{1} - i v_{2} \\
  v_{1} + i v_{2} & 1 - v_{3} \\
\end{array}
\right),
\end{displaymath}
where $\mathbf{v} = ( v_{1} \ v_{2} \ v_{3} )^{T} \in \mathbb{R}^{3}$. So the qubit state $\rho_{\mathbf{v}}$ is parameterized by some $\mathbf{v} \in \mathbb{R}^{3}$ such that $\Vert \mathbf{v} \Vert \leq 1$. In this case the vector $\mathbf{v}$ is known as the Bloch vector or Bloch vector representation of the
state $\rho_{\mathbf{v}}$.

\begin{remark} \label{R:eig}
Via a characteristic equation of the qubit state $\rho_{\mathbf{v}}$, we obtain that
its eigenvalues are
\begin{displaymath}
\frac{1 + \Vert \mathbf{v} \Vert}{2}, \ \frac{1 - \Vert \mathbf{v} \Vert}{2},
\end{displaymath}
and its determinant is
\begin{displaymath}
\det{\rho_{\mathbf{v}}} = \frac{1 - \Vert \mathbf{v} \Vert^{2}}{4} = \frac{1}{4 \gamma_{\mathbf{v}}^{2}}.
\end{displaymath}
So the mixed states are parameterized by the open unit ball $\mathbf{B}$ in $\mathbb{R}^{3}$, meanwhile the pure states are parameterized by the boundary of $\mathbf{B}$, the unit sphere (Bloch sphere). See \cite{Kim, Un08} for more information.
\end{remark}

It has been shown in \cite{Kim, Kim16} that the map
\begin{displaymath}
\rho : (\mathbf{B}, \oplus_{E}, \otimes) \to (\mathbb{D}_{2}, \odot, \star), \ \mathbf{v} \mapsto \rho_{\mathbf{v}}
\end{displaymath}
is an isomorphism between two gyrovector spaces. On the gyrovector space $\mathbb{D}_{2}$, moreover, we have that the identity is $\displaystyle \frac{1}{2} I_{2}$ and the inverse for $\rho_{\mathbf{u}}$ is
\begin{equation} \label{E:inverse}
\displaystyle \rho_{- \mathbf{u}} = \frac{1}{4 \gamma_{\mathbf{u}}}
\rho_{\mathbf{u}}^{-1}.
\end{equation}

\begin{proposition} \label{P:2-mean}
For any $\mathbf{u}, \mathbf{v} \in \mathbf{B}$
\begin{displaymath}
\rho_{\mathbf{u}} \#_{t} \rho_{\mathbf{v}} = L_{1-t}(\mu) \left( \frac{\gamma_{\mathbf{u}}}{\gamma_{\mathbf{v}}} \right)^{t} \rho_{\mathbf{u}} + L_{t}(\mu) \left( \frac{\gamma_{\mathbf{v}}}{\gamma_{\mathbf{u}}} \right)^{1-t} \rho_{\mathbf{v}}, \ t \in [0,1]
\end{displaymath}
where $\mu = \gamma_{\mathbf{u}} (1 - \mathbf{u}^{T} \mathbf{v}) (1 \pm \Vert \mathbf{u} \oplus_{E} (- \mathbf{v}) \Vert)$.
\end{proposition}

\begin{proof}
Let $A = 2 \gamma_{\mathbf{u}} \rho_{\mathbf{u}}$ and $B = 2 \gamma_{\mathbf{v}} \rho_{\mathbf{v}}$. By Remark \ref{R:eig} one can see that $A, B$ are $2 \times 2$ positive definite matrices with determinant $1$. By the joint homogeneity of two-variable geometric mean and Lemma \ref{L:2-mean}
\begin{displaymath}
\begin{split}
\rho_{\mathbf{u}} \#_{t} \rho_{\mathbf{v}} & = \left( \frac{1}{2 \gamma_{\mathbf{u}}} A \right) \#_{t} \left( \frac{1}{2 \gamma_{\mathbf{v}}} B \right) = \frac{1}{2 \gamma_{\mathbf{u}}^{1-t} \gamma_{\mathbf{v}}^{t}} A \#_{t} B \\
& = \frac{1}{2 \gamma_{\mathbf{u}}^{1-t} \gamma_{\mathbf{v}}^{t}} \left[ L_{1-t}(\mu) A + L_{t}(\mu) B \right] = L_{1-t}(\mu) \left( \frac{\gamma_{\mathbf{u}}}{\gamma_{\mathbf{v}}} \right)^{t} \rho_{\mathbf{u}} + L_{t}(\mu) \left( \frac{\gamma_{\mathbf{v}}}{\gamma_{\mathbf{u}}} \right)^{1-t} \rho_{\mathbf{v}}.
\end{split}
\end{displaymath}
Here, $\mu$ is an eigenvalue of $AB^{-1}$. Note that $\displaystyle AB^{-1} = \frac{\gamma_{\mathbf{u}}}{\gamma_{\mathbf{v}}} \rho_{\mathbf{u}} \rho_{\mathbf{v}}^{-1} = 4 \gamma_{\mathbf{u}} \rho_{\mathbf{u}} \rho_{- \mathbf{v}}$ by \eqref{E:inverse}, and $\rho_{\mathbf{u}} \rho_{- \mathbf{v}}$ is similar to
\begin{displaymath}
\rho_{\mathbf{u}}^{\frac{1}{2}} \rho_{- \mathbf{v}} \rho_{\mathbf{u}}^{\frac{1}{2}} = \tr (\rho_{\mathbf{u}} \rho_{- \mathbf{v}}) \rho_{\mathbf{u} \oplus_{E} (- \mathbf{v})} = \frac{1 - \mathbf{u}^{T} \mathbf{v}}{2} \rho_{\mathbf{u} \oplus_{E} (- \mathbf{v})}.
\end{displaymath}
The first equality follows from the gyrogroup isomorphism in \cite[Theorem 3.4]{Kim}. Thus, eigenvalues of $AB^{-1}$ are $\gamma_{\mathbf{u}} (1 - \mathbf{u}^{T} \mathbf{v}) (1 \pm \Vert \mathbf{u} \oplus_{E} (- \mathbf{v}) \Vert)$.
\end{proof}

\begin{remark} \label{R:mu}
There is an important and useful connection between the Lorentz factor of the Einstein vector sum and the Lorentz factors of the summands:
\begin{equation} \label{E:gamma-identity}
\gamma_{\mathbf{u} \oplus_{E} \mathbf{v}} = \gamma_{\mathbf{u}} \gamma_{\mathbf{v}} (1 + \mathbf{u}^{T} \mathbf{v}).
\end{equation}
Since $\gamma_{\mathbf{v}} = \gamma_{- \mathbf{v}}$, the value $\mu$ in Proposition \ref{P:2-mean} can be rewritten as
\begin{displaymath}
\mu = \frac{\gamma_{\mathbf{u} \oplus_{E} (- \mathbf{v})}}{\gamma_{\mathbf{v}}} (1 \pm \Vert \mathbf{u} \oplus_{E} (- \mathbf{v}) \Vert) = \frac{1}{\gamma_{\mathbf{v}}} \left[ \frac{1 \pm \Vert \mathbf{u} \oplus_{E} (- \mathbf{v}) \Vert}{1 \mp \Vert \mathbf{u} \oplus_{E} (- \mathbf{v}) \Vert} \right]^{1/2}.
\end{displaymath}
For the Einstein gyrogroup $(\mathbf{B}, \oplus_{E})$, A. Ungar considers what we call the Ungar gyrometric defined by
\begin{displaymath}
\varrho(\mathbf{u}, \mathbf{v}) = \Vert -\mathbf{u} \oplus_{E} \mathbf{v} \Vert = \Vert \mathbf{u} \oplus_{E} (-\mathbf{v}) \Vert.
\end{displaymath}
He also defines what we call the rapidity metric by $d_{E}(\mathbf{u}, \mathbf{v}) = \tanh^{-1} \varrho(\mathbf{u}, \mathbf{v})$. It is known as the Cayley-Klein metric on the Beltrami-Klein model of hyperbolic geometry (see \cite{FS05}), or the Bergman metric on the symmetric structure $\mathbf{B}$ with symmetries $S_{\mathbf{w}}(\mathbf{v}) = \mathbf{u} \oplus (-\mathbf{v})$ for some $\mathbf{u} = 2 \otimes \mathbf{w}$ (see \cite{Zhu}). Since $\displaystyle \tanh^{-1} t = \frac{1}{2} \ln \frac{1 + t}{1 - t}$ for $|t| < 1$, the value $\mu$ in Proposition \ref{P:2-mean} can be rewritten as
\begin{displaymath}
\mu = \frac{1}{\gamma_{\mathbf{v}}} e^{\pm d_{E}(\mathbf{u}, \mathbf{v})}.
\end{displaymath}
\end{remark}

Applying Proposition \ref{P:2x2 sp-geomean}, \eqref{E:inverse} and \eqref{E:gamma-identity} to $A = 2 \gamma_{\mathbf{u}} \rho_{\mathbf{u}}$ and $B = 2 \gamma_{\mathbf{v}} \rho_{\mathbf{v}}$, we obtain the explicit formula of the weighted spectral geometric mean for $2 \times 2$ invertible density matrices.
\begin{proposition}
For any $\mathbf{u}, \mathbf{v} \in \mathbf{B}$
\begin{displaymath}
\rho_{\mathbf{u}} \natural_{t} \rho_{\mathbf{v}} = \frac{2^{1+t} \gamma_{\mathbf{u}} (\rho_{- \mathbf{u}} + \gamma_{\mathbf{v}} \rho_{\mathbf{v}})^{t} \rho_{\mathbf{u}} (\rho_{- \mathbf{u}} + \gamma_{\mathbf{v}} \rho_{\mathbf{v}})^{t}}{(1 + \gamma_{\mathbf{u} \oplus_{E} \mathbf{v}})^{t}}, \ t \in [0,1]
\end{displaymath}
\end{proposition}

\begin{proposition} \label{P:vectors}
For any $\mathbf{u}, \mathbf{v} \in \mathbf{B}$,
\begin{displaymath}
\frac{1 + \Vert \mathbf{m} \Vert}{1 - \Vert \mathbf{m} \Vert} \leq \left[ \frac{( 1 + \Vert \mathbf{u} \Vert )( 1 + \Vert \mathbf{v} \Vert )}{( 1 - \Vert \mathbf{u} \Vert )( 1 + \Vert \mathbf{v} \Vert )} \right]^{1/2},
\end{displaymath}
where $\displaystyle \mathbf{m} = \frac{\gamma_{\mathbf{u}} \mathbf{u} + \gamma_{\mathbf{v}} \mathbf{v}}{\gamma_{\mathbf{u}} + \gamma_{\mathbf{v}}}$.
\end{proposition}

\begin{proof}
Let $A = 2 \gamma_{\mathbf{u}} \rho_{\mathbf{u}}$ and $B = 2 \gamma_{\mathbf{v}} \rho_{\mathbf{v}}$. Then $A, B \in \mathbb{P}_{2}$ and their determinants are $1$ by Remark \ref{R:eig}. Moreover,
\begin{displaymath}
\Vert A \Vert = 2 \gamma_{\mathbf{u}} \cdot \frac{1 + \Vert \mathbf{u} \Vert}{2} = \gamma_{\mathbf{u}} (1 + \Vert \mathbf{u} \Vert) = \sqrt{\frac{1 + \Vert \mathbf{u} \Vert}{1 - \Vert \mathbf{u} \Vert}},
\end{displaymath}
where $\Vert \cdot \Vert$ is the operator norm. Similarly $\Vert B \Vert = \gamma_{\mathbf{v}} (1 + \Vert \mathbf{v} \Vert)$.
Applying Corollary \ref{C:norm} to $A$ and $B$ yields
\begin{displaymath}
\Vert \gamma_{\mathbf{u}} \rho_{\mathbf{u}} + \gamma_{\mathbf{v}} \rho_{\mathbf{v}} \Vert^{2} \leq \det ( \gamma_{\mathbf{u}} \rho_{\mathbf{u}} + \gamma_{\mathbf{v}} \rho_{\mathbf{v}} ) \gamma_{\mathbf{u}} (1 + \Vert \mathbf{u} \Vert) \gamma_{\mathbf{v}} (1 + \Vert \mathbf{v} \Vert).
\end{displaymath}
Since $\gamma_{\mathbf{u}} \rho_{\mathbf{u}} + \gamma_{\mathbf{v}} \rho_{\mathbf{v}} = (\gamma_{\mathbf{u}} + \gamma_{\mathbf{v}}) \rho_{\mathbf{m}}$, it reduces to
\begin{displaymath}
\Vert \rho_{\mathbf{m}} \Vert^{2} \leq \det \left( \rho_{\mathbf{m}} \right) \gamma_{\mathbf{u}} (1 + \Vert \mathbf{u} \Vert) \gamma_{\mathbf{v}} (1 + \Vert \mathbf{v} \Vert).
\end{displaymath}
Since $\Vert \rho_{\mathbf{m}} \Vert = \frac{1 + \Vert \mathbf{m} \Vert}{2}$ and $\det \rho_{\mathbf{m}} = \frac{1 - \Vert \mathbf{m} \Vert^{2}}{4}$, we obtain the desired inequality from a simplification of the above inequality.
\end{proof}

\begin{remark}
In Theorem 4.2, \cite{KKL}, it has been shown that
\begin{displaymath}
\displaystyle \frac{\rho_{\mathbf{u}} \# \rho_{\mathbf{v}}}{\tr ( \rho_{\mathbf{u}} \# \rho_{\mathbf{v}} )} = \frac{\gamma_{\mathbf{u}} \rho_{\mathbf{u}} + \gamma_{\mathbf{v}} \rho_{\mathbf{v}}}{\gamma_{\mathbf{u}} + \gamma_{\mathbf{v}}} = \rho_{\mathbf{m}}
\end{displaymath}
for any $\mathbf{u}, \mathbf{v} \in \mathbf{B}$, where 
\begin{displaymath}
\mathbf{m} = \frac{\gamma_{\mathbf{u}} \mathbf{u} + \gamma_{\mathbf{v}} \mathbf{v}}{\gamma_{\mathbf{u}} + \gamma_{\mathbf{v}}} = \frac{1}{2} \otimes (\mathbf{u} \boxplus_{E} \mathbf{v}) = L \left( \frac{1}{2}; \mathbf{u}, \mathbf{v} \right)
\end{displaymath}
is a gyromidpoint on the Einstein gyrovector space $(\mathbf{B}, \oplus_{E}, \otimes)$. See Section 6.22 in \cite{Un08} for more information about the gyromidpoint and gyrocentroid. So Theorem \ref{T:gyrolines2} and Proposition \ref{P:2-mean} give us generalizations of the result in Theorem 4.2, \cite{KKL}.
Moreover, the inequality in Proposition \ref{P:vectors} can be rewritten as
\begin{displaymath}
2 d_{E} (\mathbf{0}, \mathbf{m}) \leq d_{E} (\mathbf{0}, \mathbf{u}) + d_{E} (\mathbf{0}, \mathbf{v}),
\end{displaymath}
where $d_{E}(\mathbf{u}, \mathbf{v}) = \tanh^{-1} \Vert -\mathbf{u} \oplus_{E} \mathbf{v} \Vert$ is the rapidity metric on the Einstein gyrovector space $(\mathbf{B}, \oplus_{E}, \otimes)$: see Remark \ref{R:mu}.
\end{remark}

\section{Final remarks}

In this section we investigate the semi-metric $d$ given by
\begin{equation} \label{E:semimetric}
\displaystyle d(A, B) = 2 \Vert \log (A^{-1} \# B) \Vert_{2}
\end{equation}
on the open convex cone $\mathbb{P}_{n}$ of positive definite matrices. Even though the semi-metric $d$ in \eqref{E:semimetric} is defined by the Frobenius norm unlike Lemma \ref{L:semimetric} with the operator norm, all properties in Proposition \ref{P:semimetric} and Theorem \ref{T:semimetric} are satisfied for $d$. The interesting consequence in the geometric point of view is that the geometric mean $A \# B$ of $A, B \in \mathbb{P}_{n}$ is the unique midpoint with respect to the Riemannian trace metric $\delta(A, B) = \Vert \log A^{-1/2} B A^{-1/2} \Vert_{2}$, meanwhile the spectral geometric mean $A \natural B$ is the midpoint with respect to the semi-metric $d$.

\begin{remark}
The Riemannian trace metric $\delta$ on the cone $\mathbb{P}$ can be rewritten as
\begin{displaymath}
\displaystyle \delta(A, B) = \Vert \log (\ominus A \oplus B) \Vert_{2}.
\end{displaymath}
Meanwhile, the semimetric in Lemma \ref{L:semimetric} can be rewritten as
\begin{displaymath}
\displaystyle d(A, B) = \Vert \log (\boxminus A \boxplus B) \Vert_{2},
\end{displaymath}
where $\boxminus A = \ominus A = (-1) \circ A$. Note that the Riemannian metric $\delta$ and the semimetric $d$ are related with the gyrogroup operation and cooperation on $\mathbb{P}$, respectively.
\end{remark}

We also compare with the semi-metric $d$ and the Riemannian trace metric $\delta$ on the open convex cone $\mathbb{P}_{n}$. First, we review the concepts of majorization. For Hermitian matrices $H$, let $\lambda(H) = ( \lambda_{1}(H), \dots, \lambda_{n}(H) )$ be the vector of all eigenvalues of $H$ such that $\lambda_{1}(H) \geq \cdots \geq \lambda_{n}(H)$, and let $| \lambda(H) |$ be the vector consisting of absolute values of eigenvalues of $H$ arranging in decreasing order. For Hermitian matrices $H$ and $K$, we say that $\lambda(K)$ \emph{weakly majorizes} $\lambda(H)$, written as $\lambda(H) \prec_{w} \lambda(K)$, if and only if
\begin{displaymath}
\sum_{i=1}^{k} \lambda_{i}(H) \leq \sum_{i=1}^{k} \lambda_{i}(K), \ k = 1, \dots, n.
\end{displaymath}
If the equality holds for $k = n$, in addition, then we say that $\lambda(K)$ \emph{majorizes} $\lambda(H)$, written as $\lambda(H) \prec \lambda(K)$. For positive semi-definite matrices $A$ and $B$, we say that $\lambda(B)$ \emph{weakly log-majorizes} $\lambda(A)$, written as $\lambda(A) \prec_{w \log} \lambda(B)$, if and only if
\begin{displaymath}
\prod_{i=1}^{k} \lambda_{i}(H) \leq \prod_{i=1}^{k} \lambda_{i}(K), \ k = 1, \dots, n.
\end{displaymath}
If the equality holds for $k = n$, in addition, then we say that $\lambda(B)$ \emph{log-majorizes} $\lambda(A)$, written as $\lambda(A) \prec_{\log} \lambda(B)$. So $\lambda(A) \prec_{w \log} \lambda(B)$ (or $\lambda(A) \prec_{\log} \lambda(B)$) for $A, B \in \mathbb{P}$ implies that $\lambda(\log A) \prec_{w} \lambda(\log B)$ (respectively, or $\lambda(\log A) \prec \lambda(\log B)$).

\begin{theorem}
For any $A, B \in \mathbb{P}_{n}$,
\begin{displaymath}
d(A, B) \leq \delta(A, B).
\end{displaymath}
The equality holds for commuting matrices $A, B \in \mathbb{P}_{n}$.
\end{theorem}

\begin{proof}
It has been shown in \cite{MA} that
\begin{displaymath}
\lambda(A \#_{t} B) \prec_{\log} \lambda(A^{1-t} B^{t})
\end{displaymath}
for any $A, B \in \mathbb{P}_{n}$ and $t \in [0,1]$. So
\begin{displaymath}
\lambda(\log (A^{-1} \# B)) \prec \lambda(\log (A^{-1/2} B^{1/2}))
\end{displaymath}
By Corollary 10.1 in \cite{Zh} and the fact that the singular values of Hermitian matrices are the absolute values of their eigenvalues,
\begin{displaymath}
s(\log (A^{-1} \# B)) = | \lambda(\log (A^{-1} \# B)) | \prec_{w} | \lambda(\log (A^{-1/2} B^{1/2})) | = s(\log (A^{-1/2} B^{1/2})),
\end{displaymath}
where $s(H)$ denotes the vector of singular values of a Hermitian matrix $H$ in decreasing order. By the Fan Dominance Theorem in \cite[Theorem IV.2.2]{Bha},
\begin{displaymath}
||| \log (A^{-1} \# B) ||| \leq ||| \log (A^{-1/2} B^{1/2}) |||
\end{displaymath}
for any unitarily invariant norm $||| \cdot |||$. Since the Frobenius norm $|| \cdot ||_{2}$ is unitarily invariant and the metric convexity of Riemannian trace metric in \cite[Proposition 6.1.10]{Bh}, we obtain
\begin{displaymath}
d(A, B) = 2 \Vert \log (A^{-1} \# B) \Vert_{2} \leq 2 \Vert \log (A^{-1/2} B^{1/2}) \Vert_{2} = 2 \delta(A^{1/2}, B^{1/2}) \leq \delta(A, B).
\end{displaymath}

Moreover, any commuting matrices $A, B \in \mathbb{P}_{n}$ are simultaneously diagonalizable by Theorem 1.3.12 in \cite{HJ}. So
\begin{displaymath}
d(A, B) = \Vert \log A - \log B \Vert_{2} = \delta(A, B).
\end{displaymath}
\end{proof}

According to above remarks and properties of the semi-metric, it is an interesting symmetric divergence. So it would be a good project to find other geometric properties of the semi-metric with the spectral geometric mean such as the metric convexity and the extension to multi-variable spectral geometric means.

\vspace{4mm}

\textbf{Acknowledgement} \\

This work was supported by the National Research Foundation of Korea (NRF) grant funded by the Korea government (No. NRF-2018R1C1B6001394).

\end{document}